\documentclass[11pt]{amsart}
\usepackage{amsmath,amssymb,amsthm}
\usepackage[a4paper,margin=1.25in]{geometry}
\usepackage[colorlinks=true,linkcolor=blue,citecolor=blue,urlcolor=blue]{hyperref}
\numberwithin{equation}{section}
\newtheorem{theorem}{Theorem}[section]
\newtheorem{proposition}[theorem]{Proposition}
\newtheorem{lemma}[theorem]{Lemma}
\newtheorem{problem}[theorem]{Problem}
\theoremstyle{remark}
\newtheorem{remark}[theorem]{Remark}

\title[Finite zero products of Toeplitz operators]{Finite zero products of Toeplitz operators with harmonic symbols}
\author{Youqi Liu,  Jie Qin, Sui Huang}
\address{School of Mathematics and Statistics, Chongqing Technology and Business University, Chongqing 400067, China}
\email{qinjie24520@163.com}
\date{}

\begin{document}
\begin{abstract}
We study finite products of Toeplitz operators with harmonic symbols on
the Fock space over the complex plane. If every symbol has the form
$f+\overline g$, where $f$ and $g$ are entire functions of exponential
type, we prove that a finite product vanishes if and only if one symbol
vanishes identically. It suffices to test the product on reproducing
kernels. The proof represents the product by a
two-variable entire function and transforms the equation for a left
factor by $\exp(-\partial_z\partial_w)$. A Fock norm identity then
permits cancellation of each nonzero factor.
\end{abstract}
\maketitle

\section{Introduction}\label{S1}

In 1964, Brown and Halmos \cite[Corollary 1 to Theorem 8]{BrownHalmos1964} proved that, for
bounded symbols $\varphi,\psi$ on the unit circle $\mathbb T$, the identity
$T_\varphi T_\psi=0$ on the Hardy space $H^2(\mathbb T)$ forces one
symbol to vanish almost everywhere. Replacing two factors by finitely
many leads to the following zero-product problem.

\begin{problem}\label{Q1}
Let $N\ge2$ and $\varphi_1,\ldots,\varphi_N\in L^\infty(\mathbb T)$.
If
$$
T_{\varphi_1}\cdots T_{\varphi_N}=0
\quad\text{on }H^2(\mathbb T),
$$
must $\varphi_j=0$ almost everywhere for some $j$?
\end{problem}

Aleman and Vukoti\'c \cite{AlemanVukotic2009} answered
Problem \ref{Q1} affirmatively for every finite $N$. The corresponding
question on Bergman spaces has led to results for several classes of
symbols. On the Bergman space $A^2(\mathbb D)$ of the unit disk, Ahern
and \v{C}u\v{c}kovi\'c \cite{AhernCuckovic2001} proved that, if
$\varphi$ and $\psi$ are bounded harmonic functions, then
$$
T_\varphi T_\psi=0
\quad\Longrightarrow\quad
\varphi\equiv0\ \text{or}\ \psi\equiv0.
$$
Further results on two-factor products were obtained through the
Berezin transform in \cite{Ahern2004}. On the unit ball $\mathbb B_n$,
Choe and Koo \cite[Theorems 1.2 and 1.3]{ChoeKoo2006} proved the
finite-product conclusion for $n+3$ harmonic symbols that are H\"older
continuous on $\overline{\mathbb B_n}$. For $n+2$ bounded harmonic
symbols, a H\"older condition near a single boundary point suffices.
For radial symbols, Le \cite{Le2010} used a refinement of Luecking's
finite-rank theorem \cite{Luecking2008} to prove that a finite product
on $A^2(\mathbb D)$ has finite rank only if one symbol vanishes,
provided one symbol belongs to $L^2(\mathbb D)$ and all the remaining
symbols are bounded and radial.

We consider the zero-product problem on the Fock space
$$
\mathcal F^2(\mathbb C)
=\left\{h\in\mathcal O(\mathbb C):
\int_{\mathbb C}|h(z)|^2\,d\mu(z)<\infty\right\},
\quad d\mu(z)=\pi^{-1}e^{-|z|^2}\,dA(z),
$$
where $dA$ is area measure. This space has reproducing kernel
$K_a(z)=e^{z\overline a}$, and we write $P$ for the orthogonal
projection from $L^2(d\mu)$ onto $\mathcal F^2(\mathbb C)$
\cite{Bargmann1961,Folland1989,Zhu2012}. The Toeplitz operator with
measurable symbol $\varphi$ is
$$
T_\varphi h=P(\varphi h),\qquad
\operatorname{Dom}(T_\varphi)
=\{h\in\mathcal F^2(\mathbb C):\varphi h\in L^2(d\mu)\}.
$$
The finite-rank theorem of Borichev and Rozenblum
\cite{BorichevRozenblum2015} implies that a bounded function symbol
cannot define a nonzero finite-rank Toeplitz operator. For two-factor
products, Bauer and Le \cite[Corollary 3.6]{BauerLe2011} proved that
$T_\varphi T_\psi=0$ forces $\varphi=0$ or $\psi=0$ when both symbols
are bounded and one is radial. Their conclusion does not extend to
three factors. In \cite[Proposition 3.7]{BauerLe2011}, they constructed
the bounded radial symbols
$$
\varphi_j(z)=|z|^{2j}e^{-|z|^2}\sin(2\sqrt3\,|z|^2),
\qquad j=0,1,2,
$$
for which each $T_{\varphi_j}$ is nonzero but
$T_{\varphi_0}T_{\varphi_1}T_{\varphi_2}=0$.

For harmonic symbols, boundedness on the whole plane permits only
constants. We instead consider $\varphi=f+\overline g$, where $f,g$
are entire functions of exponential type. Write
$$
H(1)=\{h\in\mathcal O(\mathbb C):
|h(z)|\le Ce^{r|z|}\ \text{for some }C,r>0\}.
$$
Bauer, Choe, and Koo \cite[Theorem 4.2]{Bauer2015} proved that, for two
such symbols, $T_\varphi T_\psi=T_\psi T_\varphi=0$ implies
$\varphi\equiv0$ or $\psi\equiv0$. We prove the finite-product
conclusion for this harmonic class without a commutativity assumption.
We use the common invariant space in \cite[Section 2.2]{Bauer2015}.
For $c>0$, let
$$
\mathcal D_c=\{u\text{ measurable on }\mathbb C:
ue^{-c|\cdot|^2}\in L^\infty(\mathbb C)\},
\qquad
\operatorname{Sym}(\mathbb C)=\bigcap_{c>0}\mathcal D_c,
$$
and set
$$
\mathcal H(\mathbb C)
=\left(\bigcup_{0<c<1/2}\mathcal D_c\right)\cap\mathcal O(\mathbb C).
$$
This is a dense subspace of $\mathcal F^2(\mathbb C)$, contains all
reproducing kernels, and is invariant under each $T_\varphi$ with
$\varphi\in\operatorname{Sym}(\mathbb C)$
\cite[Section 2.2]{Bauer2015}. The inequality
$$
e^{r|z|}\le e^{r^2/(4c)}e^{c|z|^2},\qquad c,r>0,
$$
shows that our harmonic symbols belong to $\operatorname{Sym}(\mathbb C)$.
Finite products are initially defined on $\mathcal H(\mathbb C)$.
A product vanishing on this dense subspace has a unique continuous
extension equal to zero on $\mathcal F^2(\mathbb C)$.
We interpret zero-product identities on $\mathcal F^2(\mathbb C)$  
in this sense.

\begin{samepage}
\begin{theorem}\label{T1}
Let $N$ be a positive integer and
$$
\varphi_j=f_j+\overline{g_j},\qquad
f_j,g_j\in H(1),\qquad 1\le j\le N.
$$
The following conditions are equivalent.
\begin{enumerate}
\item\label{T1a}
$T_{\varphi_1}\cdots T_{\varphi_N}=0$
on $\mathcal F^2(\mathbb C)$.
\item\label{T1b} $T_{\varphi_1}\cdots T_{\varphi_N}K_a=0$ for every $a\in\mathbb C$.
\item\label{T1c} $\varphi_j\equiv0$ for some $j$.
\end{enumerate}
\end{theorem}
\end{samepage}

The proof keeps the independent kernel parameter $w$ in
$$
(T_{\varphi_1}\cdots T_{\varphi_N}e^{\,\cdot\,w})(z)
=e^{zw}V(z,w).
$$
The function $V$ is of exponential type in the two variables jointly.
Conjugation by $\exp(-\partial_z\partial_w)$ yields a norm identity
on these functions and permits successive cancellation of nonzero
factors. The multiplication and differentiation formulas used in this
argument are related to Wick symbol calculus \cite{Berezin1971} and
Fock Toeplitz composition formulas
\cite{BergerCoburn1987,Coburn2001,Bauer2009}. Heat transforms in this
calculus are studied in \cite{BergerCoburn1994,AgborBauer2015}.
The required series and conjugation identities are proved in
Sections \ref{S2} and \ref{S3}.

\section{Exponential-type functions and Toeplitz operators}\label{S2}

For the kernel calculations, we use the smaller class $H(1)$ and its
two-variable counterpart. The differential representation involves
infinite-order operators on entire functions of exponential type
\cite{Carmichael1936,Dickson1964,Wimmer1979,Levin1996}.
For their transform interpretation, see
\cite{Martineau1963,CarlssonPradoReyes2016}. We use derivative norms
to control the two-variable series and polynomial approximations.

For $d=1,2$, let $\mathcal E_d$ be the space of entire functions $U$ on
$\mathbb C^d$ such that
$|U(z)|\le Ce^{a\sum_{j=1}^d|z_j|}$ for some $C,a>0$. Thus
$\mathcal E_1=H(1)$. For $r>0$, put
\begin{equation}\label{E1}
q_r(U)=\sup_{\alpha\in\mathbb N_0^d}
\frac{|\partial^\alpha U(0)|}{r^{|\alpha|}},
\qquad
E_r^{(d)}=\{U\in\mathcal O(\mathbb C^d):q_r(U)<\infty\}.
\end{equation}
Here $\mathbb N_0=\{0,1,\ldots\}$,
$|\alpha|=\sum_j\alpha_j$, and $\alpha!=\prod_j\alpha_j!$.
We use $q_r$ in either dimension, as determined by its argument.
The map $U\mapsto
(\partial^\alpha U(0)/r^{|\alpha|})_{\alpha\in\mathbb N_0^d}$
has inverse
$$
(a_\alpha)_\alpha\longmapsto
\sum_{\alpha\in\mathbb N_0^d}
\frac{a_\alpha r^{|\alpha|}}{\alpha!}z^\alpha.
$$
This series converges locally uniformly for every bounded sequence.
The two maps are isometries, so $E_r^{(d)}$ is a Banach space.

\begin{lemma}\label{L1}
We have $\mathcal E_d=\bigcup_{r>0}E_r^{(d)}$. If
$U\in E_r^{(d)}$ and $V\in E_s^{(d)}$, then
\begin{equation}\label{E2}
q_r(\partial^\alpha U)\le r^{|\alpha|}q_r(U),
\qquad
q_{r+s}(UV)\le q_r(U)q_s(V).
\end{equation}
Let $P_mU$ be the Taylor polynomial of $U$ containing all terms of total
degree at most $m$. For $s>r$,
\begin{equation}\label{E3}
q_s(U-P_mU)\le q_r(U)(r/s)^{m+1}.
\end{equation}
Moreover, with the product Gaussian measure
$d\mu(z_1)\cdots d\mu(z_d)$ on $\mathbb C^d$, we have
\begin{equation}\label{E4}
\|U\|_{\mathcal F^2(\mathbb C^d)}
\le e^{dr^2/2}q_r(U).
\end{equation}
\end{lemma}

\begin{proof}
The Taylor expansion gives
$$
|U(z)|\le q_r(U)
\sum_{\alpha\in\mathbb N_0^d}
\frac{r^{|\alpha|}|z^\alpha|}{\alpha!}
=q_r(U)e^{r\sum_j|z_j|}.
$$
Conversely, suppose $|U(z)|\le Ce^{a\sum_j|z_j|}$. For
$\alpha_j>0$, use the Cauchy radius $\alpha_j/a$ in the $j$th variable.
In a variable with $\alpha_j=0$, let the radius decrease to zero.
Since $k!\le k^k$ for $k\ge1$, the resulting estimate is
$$
|\partial^\alpha U(0)|
\le C\prod_{\alpha_j>0}
\alpha_j!(ea/\alpha_j)^{\alpha_j}
\le C(ea)^{|\alpha|}.
$$
This proves $U\in E_{ea}^{(d)}$.

The derivative estimate in \eqref{E2} follows from \eqref{E1}.
Leibniz's formula gives
$$
|\partial^\alpha(UV)(0)|
\le q_r(U)q_s(V)
\sum_{\beta\le\alpha}\binom{\alpha}{\beta}
r^{|\beta|}s^{|\alpha-\beta|}
=q_r(U)q_s(V)(r+s)^{|\alpha|}.
$$
For \eqref{E3}, the derivatives at zero of $U-P_mU$ vanish up to total
order $m$, and the remaining terms are bounded by
$q_r(U)(r/s)^{|\alpha|}$. For the Fock norm, use the orthonormal
basis $z^n/\sqrt{n!}$ from \cite[Proposition 2.1]{Zhu2012}, with its
parameter equal to $1$. Taking products in the two variables gives
$z^\alpha/\sqrt{\alpha!}$ when $d=2$. In either dimension, Parseval's
identity yields
$$
\|U\|_{\mathcal F^2(\mathbb C^d)}^2
=\sum_{\alpha\in\mathbb N_0^d}
\frac{|\partial^\alpha U(0)|^2}{\alpha!}
\le q_r(U)^2e^{dr^2},
$$
which proves \eqref{E4}.
\end{proof}

Throughout the paper, $D=\partial_z$ on one-variable functions, and
$g^*(z)=\overline{g(\overline z)}$. For $f,g,h\in H(1)$, the
representation needed below is
\begin{equation}\label{E5}
T_{f+\overline g}h=fh+g^*(D)h,
\qquad
g^*(D)h=\sum_{n=0}^{\infty}
\frac{\overline{g^{(n)}(0)}}{n!}h^{(n)}.
\end{equation}
Since $h\in H(1)\subseteq\operatorname{Sym}(\mathbb C)\cap\mathcal O(\mathbb C)$,
we can apply the Gaussian reproducing formula
\cite[Section 2.2, equation (2.4)]{Bauer2015} with its dimension and
Gaussian parameter equal to $1$. For the Taylor polynomial $g_m=P_mg$,
that formula gives
$$
P(\overline{g_m}h)
=\sum_{n=0}^m\frac{\overline{g^{(n)}(0)}}{n!}h^{(n)}.
$$
Choose $r,s>0$ such that $q_r(h),q_s(g)<\infty$ and
$q_s(g-g_m)\to0$, which is possible by \eqref{E3}. The derivative
estimate in \eqref{E2} gives
\begin{equation}\label{E6}
q_r(g^*(D)h)
\le\sum_{n=0}^{\infty}
\frac{|g^{(n)}(0)|}{n!}q_r(h^{(n)})
\le e^{rs}q_s(g)q_r(h).
\end{equation}
Thus the polynomial sums converge in $q_r$, hence in the Fock norm
by \eqref{E4}. On the projection side, \eqref{E2} and \eqref{E4}
give
$$
\begin{aligned}
\|P(\overline{g-g_m}h)\|_{\mathcal F^2(\mathbb C)}
&\le\|(g-g_m)h\|_{\mathcal F^2(\mathbb C)}\\
&\le e^{(r+s)^2/2}q_s(g-g_m)q_r(h)\longrightarrow0.
\end{aligned}
$$
Passing to the limit proves \eqref{E5}.

\section{A cancellation lemma}\label{S3}

We now work on $\mathcal E_2$ with independent variables $z,w$. For
$t\in\mathbb C$, define
\begin{equation}\label{E7}
J_tU=e^{t\partial_z\partial_w}U
=\sum_{k=0}^{\infty}\frac{t^k}{k!}\partial_z^k\partial_w^kU.
\end{equation}
For $U(z,w)=f(z)g^*(w)$ with $f,g\in H(1)$, restriction to
$w=\overline z$ gives the series in
\cite[Proposition 3.6]{Bauer2015}, which represents the heat transform
of $f\overline g$ for $t>0$ by \cite[Lemma 3.7]{Bauer2015}.
For general $U\in E_r^{(2)}$, \eqref{E2} gives
$q_r(\partial_z^k\partial_w^kU)\le r^{2k}q_r(U)$ for every $k\ge0$.
Thus \eqref{E7} converges absolutely in operator norm on $E_r^{(2)}$,
locally uniformly in $t\in\mathbb C$. The double series for $J_tJ_s$
is bounded in operator norm by $e^{(|t|+|s|)r^2}$.
Regrouping it by total order gives
$$
\begin{aligned}
J_tJ_sU
&=\sum_{m=0}^{\infty}
\left(\sum_{k=0}^m\frac{t^ks^{m-k}}{k!(m-k)!}\right)
\partial_z^m\partial_w^mU\\
&=\sum_{m=0}^{\infty}\frac{(t+s)^m}{m!}
\partial_z^m\partial_w^mU=J_{t+s}U.
\end{aligned}
$$
In particular,
\begin{equation}\label{E8}
q_r(J_tU)\le e^{|t|r^2}q_r(U),
\qquad
J_tJ_s=J_{t+s},
\qquad
J_t^{-1}=J_{-t}.
\end{equation}

For $G\in H(1)$ and $U\in\mathcal E_2$, set
\begin{equation}\label{E9}
N_GU(z,w)=\sum_{k=0}^{\infty}
\frac{G^{(k)}(w)}{k!}\partial_z^kU(z,w).
\end{equation}
For $q_s(G),q_r(U)<\infty$, \eqref{E2} gives
\begin{equation}\label{E10}
\begin{aligned}
q_{r+s}(N_GU)
&\le\sum_{k=0}^{\infty}\frac{q_s(G^{(k)})q_r(\partial_z^kU)}{k!}\\
&\le q_s(G)q_r(U)\sum_{k=0}^{\infty}\frac{(rs)^k}{k!}
=e^{rs}q_s(G)q_r(U).
\end{aligned}
\end{equation}
This proves absolute convergence in $E_{r+s}^{(2)}$ and
$N_G\mathcal E_2\subseteq\mathcal E_2$.

For $F\in\mathcal E_2$, let $M_F:\mathcal E_2\to\mathcal E_2$ denote
multiplication by $F$. For a polynomial $U$,
Leibniz's formula in \eqref{E7} gives
$$
\begin{aligned}
J_t(wU)&=wJ_tU+t\partial_zJ_tU,\\
J_t(zU)&=zJ_tU+t\partial_wJ_tU.
\end{aligned}
$$
Combining these identities with \eqref{E8}, we obtain
\begin{equation}\label{E21}
J_tM_wJ_{-t}=M_w+t\partial_z,
\qquad J_tM_zJ_{-t}=M_z+t\partial_w
\end{equation}
on polynomials. If $G$ is also a polynomial, the commuting operators
$M_w$ and $\partial_z$ give
$$
J_1M_{G(w)}J_{-1}U
=G(M_w+\partial_z)U
=\sum_{k=0}^{\deg G}\frac{G^{(k)}(w)}{k!}\partial_z^kU
=N_GU.
$$
For general $G\in H(1)$ and $U\in\mathcal E_2$, choose Taylor
polynomials $G_m,U_m$ such that $G_m\to G$ in $q_s$ and
$U_m\to U$ in $q_r$. This is possible after increasing $r,s$ in
\eqref{E3}. The estimate \eqref{E10} yields
$$
q_{r+s}(N_{G_m}U_m-N_GU)
\le e^{rs}\bigl(q_s(G_m-G)q_r(U_m)
+q_s(G)q_r(U_m-U)\bigr)\longrightarrow0.
$$
Similarly, \eqref{E2} and \eqref{E8} give
$$
\begin{aligned}
&q_{r+s}\bigl(J_1M_{G_m(w)}J_{-1}U_m
-J_1M_{G(w)}J_{-1}U\bigr)\\
&\quad\le e^{r^2+(r+s)^2}
\bigl(q_s(G_m-G)q_r(U_m)+q_s(G)q_r(U_m-U)\bigr)
\longrightarrow0.
\end{aligned}
$$
The polynomial identity therefore extends to $\mathcal E_2$, and
\eqref{E8} gives
\begin{equation}\label{E11}
J_{-1}N_GJ_1=M_{G(w)}.
\end{equation}

We use the standard two-variable Fock space
$\mathcal F=\mathcal F^2(\mathbb C^2)$ with measure
$d\nu(z,w)=d\mu(z)d\mu(w)$ and inner product linear in the first variable.
Let $\Pi$ be the orthogonal projection from $L^2(d\nu)$ onto $\mathcal F$.

\begin{proposition}\label{P1}
Suppose $f,G\in H(1)$ and $f(z)+G(w)\not\equiv0$. If $V\in\mathcal E_2$
satisfies
\begin{equation}\label{E12}
f(z)V+N_GV=0,
\end{equation}
then $V=0$.
\end{proposition}

\begin{proof}
Put $W=J_{-1}V$. Equation \eqref{E11} transforms \eqref{E12} into
$(B+C)W=0$, where
$$
B=J_{-1}M_{f(z)}J_1,\qquad C=M_{G(w)}.
$$
On the common invariant space $\mathcal E_2$, define
$$
B^\sharp=J_{-1}M_{f^*(-w)}J_1,
\qquad C^\sharp=G^*(\partial_w).
$$
If $q_s(f),q_s(G),q_r(U)<\infty$, the bound for $B$ follows from
$$
\begin{aligned}
q_{r+s}(BU)
&\le e^{(r+s)^2}q_{r+s}(fJ_1U)\\
&\le e^{(r+s)^2}q_s(f)q_r(J_1U)
\le e^{r^2+(r+s)^2}q_s(f)q_r(U).
\end{aligned}
$$
Since $q_s(f^*(-w))=q_s(f)$, this also bounds $B^\sharp$.
For $C^\sharp$, the two-variable derivative estimate in \eqref{E2} gives
$$
q_r(C^\sharp U)
\le\sum_{k=0}^{\infty}\frac{|G^{(k)}(0)|}{k!}
q_r(\partial_w^kU)
\le e^{rs}q_s(G)q_r(U).
$$
Together with the product estimate for $C$, these bounds give
\begin{equation}\label{E13}
\begin{aligned}
\max\{q_{r+s}(BU),q_{r+s}(B^\sharp U)\}
&\le e^{r^2+(r+s)^2}q_s(f)q_r(U),\\
q_{r+s}(CU)&\le q_s(G)q_r(U),\\
q_r(C^\sharp U)&\le e^{rs}q_s(G)q_r(U).
\end{aligned}
\end{equation}

To establish their pairings and commutation relations, first take
polynomial symbols and vectors and put
$$
X=M_z-\partial_w,\qquad X^\sharp=\partial_z-M_w.
$$
The monomial identity
$\langle z^aw^b,z^cw^d\rangle_{\mathcal F}
=a!b!\,\delta_{ac}\delta_{bd}$ gives
$$
\langle zU,V\rangle_{\mathcal F}
=\langle U,\partial_zV\rangle_{\mathcal F},\qquad
\langle\partial_wU,V\rangle_{\mathcal F}
=\langle U,wV\rangle_{\mathcal F}.
$$
Subtracting proves
$\langle XU,V\rangle_{\mathcal F}=\langle U,X^\sharp V\rangle_{\mathcal F}$.
By \eqref{E21},
$$
J_{-1}M_zJ_1=X,\qquad
J_{-1}M_{-w}J_1=X^\sharp,
$$
so polynomial substitution gives $B=f(X)$ and
$B^\sharp=f^*(X^\sharp)$.
Since
$$
[X^\sharp,X]=[\partial_z,M_z]+[M_w,\partial_w]=0,
\qquad
[X^\sharp,M_w]=[\partial_w,X]=0,
$$
where $[S,T]=ST-TS$, we obtain
\begin{equation}\label{E14}
B^\sharp B=BB^\sharp,\qquad
B^\sharp C=CB^\sharp,\qquad
C^\sharp B=BC^\sharp.
\end{equation}
Writing $f(z)=\sum_k a_kz^k$, repeated use of the pairing for $X$
gives
$$
\langle BU,V\rangle_{\mathcal F}
=\sum_k a_k\langle X^kU,V\rangle_{\mathcal F}
=\left\langle U,\sum_k\overline{a_k}(X^\sharp)^kV\right\rangle_{\mathcal F}.
$$
The monomial pairing for $w$ gives the same conclusion for $C$.
We have proved
\begin{equation}\label{E15}
\langle BU,V\rangle_{\mathcal F}
=\langle U,B^\sharp V\rangle_{\mathcal F},
\qquad
\langle CU,V\rangle_{\mathcal F}
=\langle U,C^\sharp V\rangle_{\mathcal F}.
\end{equation}

For general symbols and vectors, \eqref{E3} allows us to choose
$r,s>0$ and Taylor polynomials with $f_m\to f$ and $G_m\to G$ in $q_s$,
and $U_m\to U$ and $V_m\to V$ in $q_r$. Let $B_m,B_m^\sharp,C_m,C_m^\sharp$ be the corresponding
operators. For any one of these operators $A_m$ and its limit $A$,
\eqref{E13} gives
$$
q_{r+s}(A_mU_m-AU)
\le C_{r,s}\bigl(\varepsilon_m q_r(U_m)+q_r(U_m-U)\bigr)
\longrightarrow0,
$$
where $\varepsilon_m=q_s(f_m-f)+q_s(G_m-G)$ and $C_{r,s}$ also depends
on $q_s(f)+q_s(G)$, but not on $m$. To pass to compositions, write
$$
B_m^\sharp B_mU_m-B^\sharp BU
=(B_m^\sharp-B^\sharp)B_mU_m+B^\sharp(B_mU_m-BU).
$$
Two applications of \eqref{E13} bound its $q_{r+2s}$ norm by
$$
C_{r,s}\bigl(\varepsilon_m q_s(f_m)q_r(U_m)
+q_{r+s}(B_mU_m-BU)\bigr)\longrightarrow0.
$$
The other compositions in \eqref{E14} have the same two-factor
decomposition. This proves \eqref{E14} on $\mathcal E_2$.
By \eqref{E4}, the single-factor limits also hold in $\mathcal F$, so
$$
\langle BU,V\rangle_{\mathcal F}
=\lim_{m\to\infty}\langle B_mU_m,V_m\rangle_{\mathcal F}
=\lim_{m\to\infty}\langle U_m,B_m^\sharp V_m\rangle_{\mathcal F}
=\langle U,B^\sharp V\rangle_{\mathcal F}.
$$
The same limit with $C_m$ proves the second identity in \eqref{E15}.

For $W\in\mathcal E_2$, \eqref{E14} and \eqref{E15} imply
$$
\begin{aligned}
\|BW\|_{\mathcal F}^2
&=\langle W,B^\sharp BW\rangle_{\mathcal F}
=\langle W,BB^\sharp W\rangle_{\mathcal F}
=\|B^\sharp W\|_{\mathcal F}^2,\\
\langle BW,CW\rangle_{\mathcal F}
&=\langle C^\sharp BW,W\rangle_{\mathcal F}
=\langle BC^\sharp W,W\rangle_{\mathcal F}
=\langle C^\sharp W,B^\sharp W\rangle_{\mathcal F}.
\end{aligned}
$$
The real parts of the cross terms are therefore equal. Expanding the
two squared norms now gives
\begin{equation}\label{E16}
\|(B+C)W\|_{\mathcal F}^2
-\|(B^\sharp+C^\sharp)W\|_{\mathcal F}^2
=\|CW\|_{\mathcal F}^2-\|C^\sharp W\|_{\mathcal F}^2.
\end{equation}
The product $\overline{G(w)}W$ belongs to $L^2(d\nu)$, since its
absolute value equals that of $G(w)W$. For every holomorphic polynomial $p$, \eqref{E15} gives
$$
\langle p,C^\sharp W\rangle_{\mathcal F}
=\langle G(w)p,W\rangle_{\mathcal F}
=\langle p,\overline{G(w)}W\rangle_{L^2(d\nu)}.
$$
Since polynomials are dense in $\mathcal F$,
$C^\sharp W=\Pi(\overline{G(w)}W)$.
It follows that
$$
\begin{aligned}
\|CW\|_{\mathcal F}^2-\|C^\sharp W\|_{\mathcal F}^2
&=\|\overline{G(w)}W\|_{L^2(d\nu)}^2
-\|\Pi(\overline{G(w)}W)\|_{\mathcal F}^2\\
&=\|(I-\Pi)(\overline{G(w)}W)\|_{L^2(d\nu)}^2.
\end{aligned}
$$
Substitution into \eqref{E16} proves
\begin{equation}\label{E17}
\begin{aligned}
\|(B+C)W\|_{\mathcal F}^2
={}&\|(B^\sharp+C^\sharp)W\|_{\mathcal F}^2\\
&+\|(I-\Pi)(\overline{G(w)}W)\|_{L^2(d\nu)}^2.
\end{aligned}
\end{equation}
Since $(B+C)W=0$, \eqref{E17} shows that $\overline{G(w)}W$ agrees
with its holomorphic projection, everywhere by continuity. Its
$\overline w$ derivative is therefore zero, so
$$
\overline{G'(w)}W(z,w)=0.
$$
If $G$ is nonconstant, then $G'$ is nonzero on a nonempty open set,
and the identity theorem gives $W=0$. If $G=c$ is constant, then
$$
B+C=J_{-1}M_{f+c}J_1.
$$
The function $f+c$ is not identically zero, so its multiplication
operator is injective. The invertibility in \eqref{E8} again gives
$W=0$. In both cases, $V=J_1W=0$.
\end{proof}

\section{Finite products}\label{S4}

For $U\in E_r^{(2)}$, the Taylor bound in Lemma \ref{L1} gives
$$
|e^{zw}U(z,w)|\le q_r(U)e^{r|w|}e^{(r+|w|)|z|}.
$$
For each fixed $w$, this places $e^{zw}U(z,w)$ in $H(1)$, so
\eqref{E5} applies. The differential term is computed as follows.

\begin{lemma}\label{L2}
For $G\in H(1)$ and $U\in\mathcal E_2$,
\begin{equation}\label{E18}
G(D)\bigl(e^{zw}U(z,w)\bigr)
=e^{zw}N_GU(z,w),\qquad z,w\in\mathbb C.
\end{equation}
\end{lemma}

\begin{proof}
Choose $r,s>0$ such that $q_r(U),q_s(G)<\infty$. The derivative
estimate in \eqref{E2} and the Taylor bound in Lemma \ref{L1} imply
$$
\begin{aligned}
&\sum_{m=0}^{\infty}\frac{|G^{(m)}(0)|}{m!}
\sum_{k=0}^{m}\binom{m}{k}|w|^{m-k}
|\partial_z^kU(z,w)|\\
&\hspace{1cm}\le
q_s(G)q_r(U)e^{r(|z|+|w|)}
\sum_{m=0}^{\infty}\frac{s^m(|w|+r)^m}{m!}\\
&\hspace{1cm}=
q_s(G)q_r(U)e^{r(|z|+|w|)+s(|w|+r)}.
\end{aligned}
$$
This bound permits the interchange in the Leibniz expansion, giving
$$
\begin{aligned}
e^{-zw}G(D)(e^{zw}U)
&=\sum_{k=0}^{\infty}\frac{\partial_z^kU}{k!}
\sum_{\ell=0}^{\infty}
\frac{G^{(k+\ell)}(0)}{\ell!}w^\ell\\
&=\sum_{k=0}^{\infty}
\frac{G^{(k)}(w)}{k!}\partial_z^kU
=N_GU.
\end{aligned}
$$
The bound is locally uniform in $(z,w)$, so all displayed series also
converge locally uniformly.
\end{proof}

\begin{proof}[Proof of Theorem \ref{T1}]
Condition (\ref{T1a}) implies (\ref{T1b}) because each $K_a$ belongs to
$\mathcal H(\mathbb C)$.
Assume (\ref{T1b}) and suppose that every $\varphi_j$ is nonzero. Put
$G_j=g_j^*$ and define, from right to left,
\begin{equation}\label{E19}
V_{N+1}=1,\qquad
V_j=f_j(z)V_{j+1}+N_{G_j}V_{j+1},
\qquad j=N,\ldots,1.
\end{equation}
Equations \eqref{E2} and \eqref{E10} show inductively that every
$V_j\in\mathcal E_2$. By \eqref{E5} and \eqref{E18},
$$
\begin{aligned}
T_{\varphi_j}(e^{zw}V_{j+1})
&=f_j(z)e^{zw}V_{j+1}+G_j(D)(e^{zw}V_{j+1})\\
&=e^{zw}\bigl(f_j(z)V_{j+1}+N_{G_j}V_{j+1}\bigr)
=e^{zw}V_j.
\end{aligned}
$$
Starting with $V_{N+1}=1$, this identity gives
\begin{equation}\label{E20}
(T_{\varphi_j}\cdots T_{\varphi_N}e^{\,\cdot\,w})(z)
=e^{zw}V_j(z,w).
\end{equation}
Since $e^{zw}=K_{\overline w}(z)$, condition (\ref{T1b}) and \eqref{E20}
give $V_1=0$.

Each function $f_j(z)+G_j(w)$ is nonzero, since its restriction to
$w=\overline z$ is $\varphi_j(z)$. Applying Proposition \ref{P1} to
\eqref{E19} first gives $V_2=0$, then $V_3=0$, and ultimately
$V_{N+1}=0$. This contradicts $V_{N+1}=1$ and proves (\ref{T1c}).
If (\ref{T1c}) holds, invariance of $\mathcal H(\mathbb C)$
shows that the product vanishes there. Its continuous extension
is therefore zero on $\mathcal F^2(\mathbb C)$, proving (\ref{T1a}).
\end{proof}

\begin{remark}\label{R1}
The injectivity in Proposition \ref{P1} concerns the action on
two-variable kernel functions. An individual Toeplitz factor need not
be injective on $\mathcal H(\mathbb C)$. For example, $T_{\overline z}=D$ annihilates the
constants. The recursion \eqref{E19} retains the independent kernel
parameter, which permits cancellation despite this kernel.
\end{remark}

\begingroup
\linespread{1}\selectfont

\endgroup
\end{document}